\documentclass[reqno]{amsart}
\usepackage[T1]{fontenc}
\usepackage{dsfont}
\usepackage{mathrsfs}
\usepackage[colorlinks]{hyperref}
\usepackage{xcolor}
\usepackage[a4paper,asymmetric]{geometry}
\usepackage{mathscinet}
\usepackage{latexsym}
\usepackage{amsthm}
\usepackage{amssymb}
\usepackage{amsfonts}
\usepackage{amsmath}
\usepackage{longtable}
\usepackage{graphicx}
\usepackage{multirow}
\usepackage{multicol}

\newtheorem{theorem}{Theorem}[section]
\newtheorem{thm}[theorem]{Theorem}

\newtheorem{lem}[theorem]{Lemma}
\newtheorem{remark}[theorem]{Remark}
\newtheorem{proposition}[theorem]{Proposition}

\newtheorem{corollary}[theorem]{Corollary}

\theoremstyle{definition}

\newtheorem{defn}[theorem]{Definition}

\theoremstyle{remark}
\numberwithin{equation}{section}

 \DeclareMathAlphabet{\mathpzc}{OT1}{pzc}{m}{it}
 \DeclareMathAlphabet{\mathsfsl}{OT1}{cmss}{m}{sl}

  \newcommand{\FH}{\mathfrak{H}}

\newcommand{\dif}{\mathrm{d}}

\newcommand{\abs}[1]{\left\vert#1\right\vert}
\newcommand{\set}[1]{\left\{#1\right\}}

\newcommand{\norm}[1]{\left\Vert#1\right\Vert}

\newcommand{\E}{\mathbb{E}}

 \newcommand{\Rnum}{\mathbb{R}}

 \newcommand{\innp}[1]{\langle {#1}\rangle}
\newcommand{\Be}{\begin{equation}}
\newcommand{\Ee}{\end{equation}}
\newcommand{\Bs}{\begin{split}}
\newcommand{\Es}{\end{split}}
\newcommand{\Bes}{\begin{equation*}}
\newcommand{\Ees}{\end{equation*}}
\newcommand{\BT}{\begin{thm}}
\newcommand{\ET}{\end{thm}}
\newcommand{\Bp}{\begin{proof}}
\newcommand{\Ep}{\end{proof}}
\newcommand{\BL}{\begin{lem}}
\newcommand{\EL}{\end{lem}}
\newcommand{\BP}{\begin{proposition}}
\newcommand{\EP}{\end{proposition}}
\newcommand{\BC}{\begin{corollary}}
\newcommand{\EC}{\end{corollary}}
\newcommand{\BR}{\begin{remark}}
\newcommand{\ER}{\end{remark}}
\newcommand{\BD}{\begin{defn}}
\newcommand{\ED}{\end{defn}}
\newcommand{\BI}{\begin{itemize}}
\newcommand{\EI}{\end{itemize}}

\allowdisplaybreaks

\begin{document}
\title[Berry-Ess\'{e}en bound of fBm-OU processes]{Berry-Ess\'{e}en bound for  the Parameter Estimation  of Fractional Ornstein-Uhlenbeck Processes
}
\author[Y. Chen]{Yong CHEN}
 \address{College of Mathematics and Information Science, Jiangxi Normal University, Nanchang, 330022, Jiangxi, China }
\email{zhishi@pku.org.cn; chenyong77@gmail.com}
\author[N. Kuang]{Nenghui KUANG}
 \address{School of Mathematics and Computing Science, Hunan University of Science and Technology, Xiangtan, 411201, Hunan, China}
\email{knh1552@163.com}
 \author[Y. Li]{Ying LI}
 \address{School of Mathematics and Computional Science, Xiangtan University, Xiangtan, 411105, Hunan, China}
 \email{liying@xtu.edu.cn}
\begin{abstract}
For an Ornstein-Uhlenbeck process driven by fractional Brownian motion with Hurst index $H\in [\frac12,\frac34]$, we show the Berry-Ess\'{e}en bound of the least squares estimator of the drift parameter based on the continuous-time observation. We use an approach based on Malliavin calculus given by Kim and Park \cite{kim 3}.
\end{abstract}
\maketitle

\section{ Introduction}\label{sec 03}  Parameter estimation questions for stochastic differential equations driven by
fractional Brownian motion (fBm), whose solutions are observed continuously or at discrete time instants, have recently experienced intensive development. 
A simplest example is the 1-dimensional fractional Ornstein-Uhlenbeck (fOU) process: 
\begin{equation}\label{fOU}
\mathrm{d} X_t= -\theta X_t\mathrm{d} t+\sigma^2 \mathrm{d}B^{H}_t,\quad X_0=0,\quad 0\le t\le T,
\end{equation} where $B^{H}_t$ is a fBm with Hurst parameter $H\in(0,\,1)$. Since the
Hurst parameter $H$ and the volatility parameter $\sigma^2$ can be estimated by quadratic variation methods or using regression methods, $H$ and $\sigma^2$ are often assumed to be fixed and known and for simplicity we take $\sigma^2=1$. 


Based on the continuous-time observation of the trajectory of the process $X$, the maximum likelihood estimator (MLE) was studied in \cite{KL 02, TV 07, BK 10, BCS 11} when $H\in (0,\,1)$. Typically, the MLE is not easily computable.
Afterwards, a least-squares (LS) approach was proposed in \cite{hn}, which is given by a ratio of two Gaussian functionals:
\begin{align}
\hat{\theta}_T&=-\frac{\int_0^T X_t\mathrm{d}X_t}{\int_0^T X_t^2\mathrm{d} t}=\theta-\frac{\int_0^T X_t\mathrm{d}B^{H}_t}{\int_0^T X_t^2\mathrm{d} t},\label{hattheta}
\end{align}where the integral with respect to $B^H$ is interpreted in the Skorohod sense (or say a divergence-type integral).
The strong consistency and asymptotic normality of $\hat{\theta}_T$ are shown for $H\in [\frac12,\,\frac34)$ in \cite{hn}, and recently, this finding is extended to the case of $H\in (0,\frac34]$ in \cite{hnz}.

Skorohod integrals are difficult to use in practice.  When $H> \frac12$, we can translate the above divergence-type integral into a Young integral modulo a correction term. But it raises a problem that the correction term relies on the unknown
parameter $\theta$ that is being estimated, which makes $\hat{\theta}_T$ cannot be computed from the trajectory of $X$. When $H< \frac12$, it is worse since  we can not  reinterpret the Skorohod integral as a Young integral which is a pathwise notion \cite{BV 17}. Therefore, strictly speaking, $\hat{\theta}_T$ is not a real estimator when $H\in (0,\frac12)\cup(\frac12,1)$. 

The above problems are often called as questions of measurability of the estimator, which motivate some authors to study more practical parameter estimates based on discrete observations (e.g. \cite{TV 07, HS 13, Neu Tin 14, BV 17, EOESV17, SV 18}). Most scholars working in this direction try to discretize the continuous-time estimators, e.g., \cite{ES 13, HS 13, Neu Tin 14, SV 18}. Recently, in \cite{ EOESV17, EV 16}, the authors show that this may not be the most fruitful idea, i.e., they find out that to discretize the continuous-time estimators will lose the estimator's interpretation as a least square optimizer and that it may be better to understand the distribution of the original process taken at discrete
observation times. In \cite{EOESV17},  the authors illustrate that ``the continuous time estimator serves only as a mathematical tool in the study of the discrete-time estimators' asymptotics."

We mention that for the non-ergodic fOU process, i.e., $-\theta>0$, in the LS estimator, the integral in (\ref{hattheta}) is interpreted as a Young integral in \cite{Bel ESOu 11, MaESOu 13}, where  both the consistency and the asymptotic distributions are shown. Moreover, when $\sigma^2$ or $(\sigma^2,\,H)$ is unknown,  the joint estimation of pairs of parameters $(\theta,\, \sigma^2)$ or the three parameters $(\theta,\, \sigma^2,\,H )$ is studied in \cite{BV 17, XZX 11, ZXZN 13}. 

In this present paper, we will avoid this more practical estimates based on discrete observations and address the question whether the Berry-Ess\'{e}en bound  of $\sqrt{T}(\hat{\theta}_T-\theta) $ can be obtained. When $H=\frac12$,  it is well known that the Berry-Ess\'{e}en bound can be shown by means of squeezing techniques, please refer to \cite{Bsh 00, Bsh 08}   and  the references therein. But the case of $H\neq \frac12$ has not  been solved up to now.  Fortunately,  two new approaches based on the Malliavin calculus are proposed to show the Berry-Ess\'{e}en bound recently \cite{kim 2, kim 3}. 
We will give a positive answer to the case of $H\in [\frac12,\,\frac34]$ using one of these two approaches (see also Theorem~\ref{kp} below). 
\begin{thm} \label{main thm}
Let $Z$ be a standard Gaussian random variable.
When $H\in [\frac12, \,\frac34)$, there exists a constant $C_{\theta, H}$ such that when $T$ is large enough, 
\begin{equation}\label{hle 34}
\sup_{z\in \Rnum}\abs{P(\sqrt{\frac{T}{\theta \sigma^2_H}} (\hat{\theta}_T-\theta )\le z)-P(Z\le z)}\le\frac{ C_{\theta, H}}{T^{\beta}},
\end{equation}where
 \begin{equation}
\beta=\left\{
      \begin{array}{ll}
 \frac12, & \quad \text{if } H\in [\frac12,\,\frac58),\\
 \frac38-,&\quad \text{if } H=\frac58,\\
3-4H, &\quad \text{if } H\in (\frac58,\, \frac34).
  \end{array}
\right.
\end{equation}
When $H=\frac34$, there exists a constant $C_{\theta}$ such that  when $T$ is large enough, 
 \begin{equation}\label{upper bound 2}
\sup_{z\in \Rnum}\abs{P(\sqrt{\frac{T}{\theta \sigma^2_H\log T}} (\hat{\theta}_T-\theta )\le z)-P(Z\le z)}\le \frac{C_{\theta}}{{\log T}}, 
\end{equation}
where $\sigma^2_{H}$ is given in \cite{hn,hnz} as follows:
\begin{equation}
\sigma^2_H=
\left\{
      \begin{array}{ll}
 (4H-1) \big( 1+\frac{\Gamma(3-4H)\Gamma(4H-1)}{\Gamma(2H)\Gamma(2-2H)}\big), & \quad H\in [\frac12, \frac34),\\
 \frac{4}{\pi}, &\quad H=\frac34.
     \end{array}
\right.
\end{equation}
\end{thm} This finding is only a first step to understand the Berry-Ess\'{e}en behavior of the estimator because of the question of measurability mentioned above. When $H\in (\frac58,\, \frac34)$, the same upper bound is obtained in \cite{JLW 19} using an alternative method. In \cite{ES 13}, a discrete time least squares estimator is proposed and an upper Berry-Ess\'{e}en-type bound in the Kolmogorov distance is  shown  when the time interval between two consecutive observations converges to 0.  If observations are fixed time step,  the so-called ``polynomial variation'' estimator is proposed and an optimal Berry-Ess\'{e}en-type bound in the Wasserstein distance is shown in \cite{EV 16}.

Proof of Theorem~\ref{main thm} will be given in Section~\ref{sec prf}.  Although  the lower bound of Kolmogorov distance between $\sqrt{T}(\hat{\theta}_T-\theta) $ and the Gaussian random variable is known in case of $H=\frac12$ \cite{kim 2}, we do not give the similar result in case of $H\neq \frac12$. 
Throughout the paper we assume $H \ge \frac12$. The case $H<\frac12$ will involve much more complex computations. We shall investigate this case separately.


\section{Preliminary} Let $\alpha_H=H(2H-1)$.
Let $\mathcal{E}$ denote the space of all real valued step functions on $[0,\,T]$. The Hilbert space $\mathfrak{H}$ is defined
as the closure of $\mathcal{E}$ endowed with the inner product
\begin{align*}
\innp{\mathbf{1}_{[0,t]},\,\mathbf{1}_{[0,s]}}_{\FH}=\E\big(B^H_t  B^H_s \big).
 \end{align*}Define the Banach space
\begin{align*}
\abs{\mathfrak{H}}=\set{ f | f:[0,\,T]\to \Rnum, \int_0^{T}\int_0^{T} \abs{f(t)f(s)}\abs{t-s}^{2H-2}\dif t\dif s<\infty}.
\end{align*}
It is well known that $L^{\frac{1}{H}}\subset \abs{\mathfrak{H}}\subset \mathfrak{H}$ and when $\varphi,\,\psi\in \abs{\mathfrak{H}}$,
\begin{align*}
\innp{\varphi,\,\psi}_{\mathfrak{H}}=\alpha_H \int_0^{T}\int_0^{T}  \varphi(t)\psi(s) \abs{t-s}^{2H-2}\dif t\dif s.
\end{align*}

The Gaussian isonormal process associated with $\mathfrak{H}$ is given by Wiener integrals with
respect to a fBm for any deterministic kernel $\in\mathfrak{H}$:
\begin{align*}
B^{H}(f)=\int_{0}^{T} f(s)\dif B^{H}_s.
\end{align*}

Let $H_n$ be the $n$-th Hermite polynomial. 
The closed linear subspace $\mathcal{H}_n$ of $L^2(\Omega)$ generated by $\set{H_n(B^H(f)):\, f \in\mathfrak{H}, \norm{f}_{\mathfrak{H}}=1} $ is called the $n$-th Wiener-Ito chaos. The linear  isometric mapping $I_n:\, \mathfrak{H}^{\odot n}\to \mathcal{H}_n$  given by $I_n(h^{\otimes n})=n! H_n(B^H(f))$ is called the $n$-th multiple Wiener-Ito integral. For any $f\in \mathfrak{H}^{\otimes n}$, define $I_n(f)=I_n(\tilde{f})$ where $\tilde{f}$ is the symmetrization of $f$.

Given $f\in \mathfrak{H}^{\otimes p}$ and $g\in \mathfrak{H}^{\otimes q}$ and $r=1,\cdots, p\wedge q$,  
$r$-th contraction between $f$ and $g$ is the element of $\mathfrak{H}^{\otimes (p+q-2r)}$ defined by
\begin{align*}
f\otimes_{r} g (t_1,\dots,t_{p+q-2r})&=\alpha^{2r}_H\int_{[0,\,T]^{2r}} \abs{u_1-v_1}^{2H-2}\dots \abs{u_r-v_r}^{2H-2} f(t_1,\dots,t_{p-r},u_1,\dots,u_r)\\
&\times g(t_{p-r+1},\dots,t_{p+q-2r},v_1,\dots,v_r)\dif \vec{u} \dif \vec{v},
\end{align*}where $\vec{u}=(u_1,\dots, u_r)$, $\vec{v}=(v_1,\dots,v_r)$.

We will make use of the following estimate of the Kolmogrov distance between a nonlinear Gaussian functional and the standard normal (see Corollary 1 of \cite{kim 3}).
\begin{thm}[Kim, Y. T., \& Park, H. S]\label{kp}
Suppose that $\varphi_T(t,s)$ and $\psi_T(t,s)$ are two functions on $\mathfrak{H}^{\otimes 2}$.
Let $b_T$ be a positive function of $T$ such that $I_2(\psi_T)+b_T>0$ a.s. If $\Psi_i(T)\to 0,\,i=1,2,3$ as $T\to \infty$, then there exists a constant $c$ such that for $T$ large enough, 
\begin{equation}
\sup_{z\in \Rnum}\abs{P(\frac{I_2(\varphi_T)}{ I_2(\psi_T)+b_T}\le z)-P(Z\le z)}\le c\times \max_{i=1,2,3} \Psi_i(T),
\end{equation}
where 
\begin{align*}
\Psi_1(T)&=\frac{1}{b_T^2}\sqrt{\big[b^2_T-2\norm{\varphi_T}_{\mathfrak{H}^{\otimes 2}}^2\big]^2+8\norm{\varphi_T \otimes_1 \varphi_T}_{\mathfrak{H}^{\otimes 2}}^2},\\
\Psi_2(T)&=\frac{2}{b_T^2}\sqrt{2\norm{\varphi_T \otimes_1 \psi_T}_{\mathfrak{H}^{\otimes 2}}^2+\innp{\varphi_T,\,\psi_T}_{\mathfrak{H}^{\otimes 2}}^2},\\
\Psi_3(T)&=\frac{2}{b_T^2}\sqrt{ \norm{\psi_T}_{\mathfrak{H}^{\otimes 2}}^4+2\norm{\psi_T \otimes_1\psi_T}_{\mathfrak{H}^{\otimes 2}}^2}.
\end{align*}
\end{thm}

\section{Proof of Theorem \ref{main thm}} \label{sec prf}
It follows from Eq.(\ref{hattheta}) and the product formula of multiple integrals that 
\begin{equation}\label{ratio 1}
\sqrt{\frac{T}{\theta \sigma^2_H}} (\hat{\theta}_T-\theta )=\frac{I_2(f_T)}{ I_2(g_T)+b_T},
\end{equation}
where
\begin{align}
f_T(t,s)&=\frac{1}{2\sqrt{\theta\sigma^2_H T}}e^{-\theta \abs{t-s}}\mathbf{1}_{\set{0\le s,t\le T}},\\
g_T(t,s)&=\sqrt{\frac{\sigma^2_H}{\theta T}}f_T-\frac{1}{2\theta T}h_T,\label{gt ts}\\
h_T(t,s)&= e^{-\theta (T-t)-\theta (T-s)}\mathbf{1}_{\set{0\le s,t\le T}},\label{ht ts}\\
b_T&=\frac{1}{T}\int_0^T\, \norm{e^{-\theta (t-\cdot) }\mathrm{1}_{[0,t]}(\cdot)}^2 _{\mathfrak{H}}\dif t.\label{bt bt}
\end{align} The reader can also refer to Eq.(17)-(19) of \cite{kim 2} for details.

We need several lemmas before the proof of Theorem~\ref{main thm}.
First, we can show the following estimate by combining a slight modification of Proposition~7 or (3.17) of \cite{hnz} with Lemma 5.4 of web-only Appendix of \cite{hn}.
\begin{lem}\label{zhou}
When $ H\in [\frac12, \frac34]$, there exists a constant $C_{\theta,\,H}$ such that 
\begin{equation}\label{zhou ineq}
\norm{f_T\otimes_{1} f_T}_{\mathfrak{H}^{\otimes 2}}\le  
C_{\theta,\,H} \left\{
      \begin{array}{ll}
 \frac{1}{\sqrt T}, & \quad H\in [\frac12, \frac58),\\
  \frac{1}{ T^{ \frac38-}}, & \quad H= \frac58,\\
 \frac{1}{  T^{3-4H}} , &\quad H\in (\frac58,\,\frac34].
     \end{array}
\right.
\end{equation}
\end{lem}
Since $H>\frac12$, we can write $b_T$ as 
\begin{align*}
b_T &=\frac{\alpha_H}{T}\int_0^T\dif t\int_{[0,t]^2}e^{-\theta (t-u)-\theta (t-v)} \abs{u-v}^{2H-2}\dif u\dif v ,\\
&=\frac{2\alpha_H}{T}\int_0^T\dif t\int_{0\le u\le v \le t}e^{-\theta (t-u)-\theta (t-v)} \abs{u-v}^{2H-2}\dif u\dif v.
\end{align*}

\begin{lem} \label{bt exponential}
When $H\ge \frac12$,  the convergent speed of $b_T \to H\Gamma(2H)\theta^{-2H}$ is at least $\frac{1}{T}$.
\end{lem}
\begin{proof}The case of $H=\frac12$ is simple. When $H>\frac12$, making change variable $ a=t-v,\,b=v-u$ and then using integration by parts, we have that
\begin{align*}
b_T &=   \frac{2\alpha_H}{T } \int_0^T\dif t\int_{0\le a+b \le t,\,a,b\ge 0}e^{-\theta (b+2a)} b^{2H-2}\dif a\dif b\\
&=  \frac{\alpha_H}{\theta T } \int_0^T\dif t\int_0^t e^{-\theta b} b^{2H-2} [1-e^{-2\theta (t-b)}] \dif b \\
&= \frac{\alpha_H}{\theta T }\Big[ \int_0^T\dif t\int_0^t e^{-\theta b} b^{2H-2}  \dif b -  \int_0^T e^{-2\theta t }\dif t\int_0^t e^{\theta b} b^{2H-2}   \dif b\Big ]\\
&= \frac{\alpha_H}{\theta } \Big\{ \int_0^T e^{-\theta t} t^{2H-2} \dif t +\frac{1}{2\theta T}\big[ e^{-2\theta T}  \int_0^Te^{\theta t} t^{2H-2} \dif t -  \int_0^T e^{-\theta t} t^{2H-2} (1+2\theta t)\dif t \big]\Big\}.
\end{align*}
Hence,  there exists a constant $C_{\theta,H}$ such that 
\begin{align*}
&\frac{\abs{b_T -H\Gamma(2H)\theta^{-2H}}}{\alpha_H/ \theta} \\
&\le  \int_T^{\infty}e^{-\theta t} t^{2H-2} \dif t +\frac{1}{2\theta Te^{2\theta T}}   \int_0^Te^{\theta t} t^{2H-2} \dif t +  \frac{1}{2\theta T}\int_0^T e^{-\theta t} t^{2H-2} (1+2\theta t)\dif t  \\
&\le \frac{C_{ \theta,H}}{T}.
\end{align*}
\end{proof}

\begin{lem}\label{ht limit}
Let $h_T$ be given as in (\ref{ht ts}). Then as $T\to \infty$,   
\begin{equation}
\frac{1}{\sqrt{T}} h_T \to 0,\quad \text{ in} \quad \mathfrak{H}^{\otimes 2}. 
\end{equation}
\end{lem}
\begin{proof} The case of $H=\frac12$ is simple. When $H>\frac12$, by the symmetrical property and the L'Hospital's rule, we have that 
\begin{align*}
& \lim_{T\to \infty}\frac{1}{\alpha_H^2{T}} \norm{h_T}_{\mathfrak{H}^{\otimes 2}}^2\\
&=\lim_{T\to \infty}\frac{1}{T}\int_{[0,T]^4}e^{-\theta[(T-t_1)+(T-s_1)+(T-t_2)+(T-s_2)]}\abs{t_1-t_2}^{2H-2}\abs{s_1-s_2}^{2H-2}\dif \vec{t}\dif \vec{s} \\
&=\lim_{T\to \infty}\frac{8}{Te^{4\theta T}}\int_{0\le t_2\le t_1\le T,\,0\le s_2\le s_1\le T,\, s_1\le t_1} e^{\theta(t_1+t_2+s_1+s_2)}\abs{t_1-t_2}^{2H-2}\abs{s_1-s_2}^{2H-2} \dif \vec{t} \dif \vec{s}\\ 
&=\lim_{T\to \infty} \frac{8}{(1+ 4\theta T)e^{3\theta T}}\int_{0\le t_2\le T,\,0\le s_2\le s_1\le T } e^{\theta( t_2+s_1+s_2)}(T-t_2)^{2H-2}(s_1-s_2)^{2H-2} \dif t _2\dif \vec{s}.
\end{align*}
We divide the domain $\set{0\le t_2\le T,\,0\le s_2\le s_1\le T,\, s_1\le T}$ into three disjoint regions according
to the distinct orders of $s1,\, s2,\, t2$: $$\Delta_1=\set{0\le s_2\le s_1\le t_2\le T},\,\Delta_2=\set{0\le s_2\le t_2\le s_1\le T},\, \Delta_3=\set{0\le t_2\le s_2\le  s_1\le T}.$$
We also denote $I_i=\int_{\Delta_i}e^{\theta( t_2+s_1+s_2 -3T)}(T-t_2)^{2H-2}(s_1-s_2)^{2H-2} \dif t_2 \dif \vec{s} $. Thus, we have that
\begin{equation}
\lim_{T\to \infty}\frac{1}{\alpha_H^2{T}} \norm{h_T}_{\mathfrak{H}^{\otimes 2}}^2=\lim_{T\to \infty} \frac{8}{1+ 4\theta T}(I_1+I_2+I_3).
\end{equation}

By making the change of variables $T-t_2=x, t_2-s_1=y,\,s_1-s_2=z$, we have that
\begin{align*}
I_1&=\int_{\Rnum_{+}^3,\, x+y+z\le T} e^{-\theta (3x+2y+z)} x^{2H-2}z^{2H-2}\dif x\dif y \dif z\\
&<\int_{\Rnum_{+}^3} e^{-\theta (3x+2y+z)} x^{2H-2}z^{2H-2}\dif x\dif y \dif z <\infty.
\end{align*}
Similarly, we can show that $I_2,\,I_3<\infty $, which implies that  $\frac{1}{ {T}} \norm{h_T}_{\mathfrak{H}^{\otimes 2}}^2\to 0 $ as $T\to \infty$.
\end{proof}

\begin{lem}\label{gt gt}
Let $g_T$ be given as in (\ref{gt ts}). When $H\in [\frac12,\,\frac34)$,  we have that as $T\to \infty$,
\begin{align}
T\norm{g_T}_{\mathfrak{H}^{\otimes 2}}^2&\to \frac{\delta_H}{2\theta ^{1+4H}},\quad
T \innp{f_T,\,g_T}_{\mathfrak{H}^{\otimes 2}}^2\to \frac{\delta_H^2}{4\theta^{1+8H} \sigma_H^2},\\
T \norm{f_T\otimes_{1}g_T}_{\mathfrak{H}^{\otimes 2}}^2&\to 0,\quad
T\norm{g_T\otimes_{1}g_T}_{\mathfrak{H}^{\otimes 2}}^2\to 0;
\end{align}
when $H=\frac34$, we have that 
\begin{align*}
\frac{T}{\log T}\norm{g_T}_{\mathfrak{H}^{\otimes 2}}^2&\to \frac{\delta_H}{2\theta ^{1+4H}},\quad
\frac{T}{\log ^2T} \innp{f_T,\,g_T}_{\mathfrak{H}^{\otimes 2}}^2\to \frac{\delta_H^2}{4\theta^{1+8H} \sigma_H^2},\\
\frac{T}{\log T}\norm{f_T\otimes_{1}g_T}_{\mathfrak{H}^{\otimes 2}}^2&\to 0,\quad
\frac{T}{\log T}\norm{g_T\otimes_{1}g_T}_{\mathfrak{H}^{\otimes 2}}^2\to 0,
\end{align*}
 where $\delta_H$ is given in \cite{hn}:
\begin{equation*}
\delta_H=\left\{
      \begin{array}{ll}
H^2(4H-1)\big( \Gamma^2(2H)+\frac{\Gamma(2H)\Gamma(3-4H)\Gamma(4H-1)}{\Gamma(2-2H)} \big),&\quad H\in [\frac12,\,\frac34),\\
\frac{9}{16},& \quad H=\frac34.
  \end{array}
\right.
\end{equation*}
\end{lem}
\begin{proof} We only show the case of $H\in[\frac12,\,\frac34)$. The case of $H=\frac34$ is similar.

It follows from (\ref{gt ts}) that 
\begin{align*}
T\norm{g_T}_{\mathfrak{H}^{\otimes 2}}^2&=\frac{\sigma^2_H}{\theta}\norm{f_T}^2_{\mathfrak{H}^{\otimes 2}}+\frac{1}{4\theta^2T} \norm{h_T}^2_{\mathfrak{H}^{\otimes 2}}-\sqrt{\frac{\sigma^2_H}{\theta^3 T}}\innp{f_T,\,h_T}_{\mathfrak{H}^{\otimes 2}}.
\end{align*} The Cauchy-Schwarz inequality implies that the third term is bounded by $\frac{c}{\sqrt{T}} \norm{f_T}\cdot\norm{h_T}$.
By Lemma~\ref{ht limit} and Eq.(3.12)-(3.14) of \cite{hn}, we have that 
\begin{align*}
\lim_{T\to \infty}T\norm{g_T}_{\mathfrak{H}^{\otimes 2}}^2&= \frac{\sigma^2_H}{\theta} \lim_{T\to \infty}\norm{f_T}^2_{\mathfrak{H}^{\otimes 2}}= \frac{\delta_H}{2\theta ^{1+4H}}.
\end{align*}

Similarly, we have that 
\begin{align*}
\lim_{T\to \infty}\sqrt{T}\innp{f_T,\,g_T}_{\mathfrak{H}^{\otimes 2}}&= \sqrt{\frac{\sigma^2_H}{\theta}}\lim_{T\to \infty}\norm{f_T}^2_{\mathfrak{H}^{\otimes 2}}= \sqrt{\frac{\theta}{\sigma^2_H}} \frac{\delta_H}{2\theta ^{1+4H}}.
\end{align*}

Next,  it is clear that 
\begin{align*}
\sqrt{T}f_T\otimes_{1}g_T=\sqrt{\frac{\sigma^2_H}{\theta}}f_T\otimes_{1}f_T -\frac{1}{2\theta} f_T\otimes_{1}( \frac{1}{\sqrt{T}}h_T).
\end{align*}
The fourth moment theorem implies that $f_T\otimes_{1}f_T\to 0$ in $\mathfrak{H}^{\otimes 2}$ as $T\to \infty$, please refer to \cite{hn, hnz} for details. The Cauchy-Schwarz inequality (or Lemma 4.2 of \cite{bie}) and Lemma~\ref{ht limit} imply that as $T\to \infty$,
\begin{align*}
\norm{f_T\otimes_{1}( \frac{1}{\sqrt{T}}h_T)}_{\mathfrak{H}^{\otimes 2}}\le \norm{f_T}_{\mathfrak{H}^{\otimes 2}}\cdot \frac{1}{\sqrt{T}}\norm{h_T}_{\mathfrak{H}^{\otimes 2}}\to 0,
\end{align*} which implies that $\sqrt{T}f_T\otimes_{1}g_T\to 0$ in $\mathfrak{H}^{\otimes 2}$. 

Finally, the Cauchy-Schwarz inequality or Lemma 4.2 of \cite{bie} implies that
\begin{align*}
\sqrt{T}\norm{g_T\otimes_{1}g_T}_{\mathfrak{H}^{\otimes 2}}\le\sqrt{T} \norm{g_T}^2_{\mathfrak{H}^{\otimes 2}}=\frac{1}{\sqrt T} \cdot T\norm{g_T}^2_{\mathfrak{H}^{\otimes 2}}\to 0.
\end{align*}
\end{proof}

\begin{lem}\label{2ft2}
When $H\in [\frac12,\,\frac{3}{4})$,  the convergence speed of $2\norm{f_T}^2_{\mathfrak{H}^{\otimes 2}}\to \big[H\Gamma(2H)\theta^{-2H}\big]^2$ is $\frac{1}{T^{3-4H}}$ as $T\to \infty$. When $H=\frac34$, the convergence speed of $\frac{2\norm{f_T}^2_{\mathfrak{H}^{\otimes 2}}}{\log T}\to \frac{9\pi}{64\theta^3}$ is $1/\log T$  as $T\to \infty$.  
\end{lem}
\begin{proof} The case of $H=\frac12$ is easy.

Next, suppose that $H\in (\frac12,\,\frac34)$.
By the symmetrical property, the L'Hospital's rule and Lemma 5.3 in the web-only Appendix of \cite{hn}, we have that as $T\to \infty$,
\begin{align*}
&\lim_{T\to \infty} {T^{3-4H}}\Big  \{-2\norm{f_T}^2_{\mathfrak{H}^{\otimes 2}}+ \big[H\Gamma(2H)\theta^{-2H}\big]^2\Big\}\times \frac{\theta \sigma^2_H}{2 \alpha_H^2}\times(4H-2)\\
&=\lim_{T\to \infty} \frac{4H-2}{ 4T^{4H-2}}\Big[-
\int_{[0,T]^4}e^{-\theta\abs{t_1-s_1}-\theta\abs{t_2-s_2}}\abs{t_1-t_2}^{2H-2} \abs{s_1-s_2}^{2H-2}\dif \vec{t}\dif \vec{s}+\frac{2\theta^{1-4H}\delta_H}{\alpha_H^2}T \Big]\\
&=\lim_{T\to \infty}  { T^{3-4H}} \Big[-\int_{[0,T]^3}e^{-\theta\abs{t_1-s_1}-\theta(T-s_2)}(T-t_1)^{2H-2} \abs{s_1-s_2}^{2H-2}\dif t_1\dif \vec{s}+\frac{\theta^{1-4H}\delta_H}{2\alpha_H^2} \Big]\\
&\quad (\text{let} \quad x=T-s_2,\,y=T-s_1,\, z=T-t_1)\\
&=\lim_{T\to \infty}  { T^{3-4H}}\Big[-\int_{[0,T]^3}e^{-\theta(x+\abs{y-z})}z^{2H-2} \abs{x-y}^{2H-2}\dif x\dif y\dif z+\frac{\theta^{1-4H}\delta_H}{2\alpha_H^2} \Big]\\
&=\lim_{T\to \infty}  { T^{3-4H}}\int_{\Rnum_{+}^3-[0,T]^3}e^{-\theta(x+\abs{y-z})}z^{2H-2} \abs{x-y}^{2H-2}\dif x\dif y\dif z\\
&:=\sum_{i=1}^6\lim_{T\to \infty}  { T^{3-4H }} I_i,
\end{align*} 
where for $\,i=1,\dots, 6$, 
\begin{align*}
I_i&= \int_{\Delta_i^c}e^{-\theta(x+\abs{y-z})}z^{2H-2} \abs{x-y}^{2H-2}\dif x\dif y\dif z,\\
\Delta_i^{c}&=\lim_{T\to \infty}\Delta_i(T) -\Delta_i(T),\\
\Delta_1(T)&=\set{0\le x\le y\le z\le T},\,\Delta_2(T)=\set{0\le x\le z\le y\le T},\, \Delta_3(T)=\set{0\le z\le x\le y\le T},\\
\Delta_4(T)&=\set{0\le y\le x\le z\le T},\,\Delta_5(T)=\set{0\le y\le z\le x\le T},\, \Delta_6(T)=\set{0\le z\le y\le  x\le T}.
\end{align*} 

By making the change of variables $a=x,\,b=y-x,\,c=z-y$, we have that 
\begin{align*}
I_1&=  \int_{\Rnum_{+}^3,\,a+b+c>T}e^{-\theta(a+c)}b^{2H-2} (a+b+c)^{2H-2}\dif a\dif b\dif c.
\end{align*}
Since on $\set{(a,b,c)\in \Rnum_{+}^3,\,a+b+c>T}$, we have that
\begin{align*}
\set{a+b+c>T,\, b\ge 1}&=\set{1\le b\le T,\,a+c> T-b}\cup\set{b>T },\\
\set{a+b+c>T,\, 0<b< 1}&\subset \set{0<b< 1,\,a+c>T-1},\\
(a+b+c)b&\ge b^2\mathbf{1}_{\set{b\ge 1}}+(a+c)b\mathbf{1}_{\set{0< b<1}}.
\end{align*}
Hence,
\begin{align*}
{ T^{3-4H }} I_{1}&={ T^{3-4H }}[ I_{11}+I_{12}+I_{13}],
\end{align*}where
\begin{align*}
I_{11}&= \int_{1}^Tb^{2H-2}   \dif b \int_{a+c>T-b}e^{-\theta(a+c)}(a+b+c)^{2H-2}\dif a \dif c \\
I_{12}&= \int_{T}^{\infty}b^{2H-2}  \dif b \int_{\Rnum_{+}^2}e^{-\theta(a+c)} (a+b+c)^{2H-2}\dif a\dif c,\\
I_{13}&=\int_{0}^1b^{2H-2}  \dif b \int_{a+c>T-1}e^{-\theta(a+c)}(a+b+c)^{2H-2}\dif a\dif c\\
&< \int_{0}^1b^{2H-2}  \dif b \int_{a+c>T-1}e^{-\theta(a+c)}(a+c)^{2H-2}\dif a\dif c.
\end{align*}
By the L'Hospital's rule and  Lebesgue's dominated convergence theorem,  we have that as $T\to \infty$
\begin{align*}
 T^{3-4H}I_{11}\to 0,\qquad T^{3-4H}I_{12} \to \frac{1}{(3-4H)\theta^2} ,\qquad  T^{3-4H}I_{13}\to 0
\end{align*}which implies that 
\begin{align*}
T^{3-4H}I_{1}\to \frac{1}{(3-4H)\theta^2}.
\end{align*}
In the same way, we have that as $T\to \infty$,
\begin{align*}
T^{3-4H}I_{2}&\to\frac{1}{(3-4H)\theta^2},\qquad T^{3-4H}I_{4}\to 0, \qquad T^{3-4H}I_{3}= T^{3-4H}I_{6}\to 0.
\end{align*}In addition, it is clear that $ 0\le I_5\le I_3$. Hence, 
\begin{align*}
\sum_{i=1}^6\lim_{T\to \infty}  { T^{3-4H }} I_i=  \frac{2}{(3-4H)\theta^2},
\end{align*} which implies
 the convergence speed of $2\norm{f_T}^2_{\mathfrak{H}^{\otimes 2}}\to \big[H\Gamma(2H)\theta^{-2H}\big]^2$ is $\frac{1}{T^{3-4H}}$ as $T\to \infty$.

Finally,  suppose that $H=\frac34$. Using L'Hospital's rule and the symmetry, we have that  
\begin{align*}
&\lim_{T\to \infty}  \log T \Big\{ \frac{2\norm{f_T}^2_{\mathfrak{H}^{\otimes 2}}}{\log T}- \big[H\Gamma(2H)\theta^{-2H}\big]^2\Big \}\times \frac{\theta \sigma^2_H}{2 \alpha_H^2}\\
&=\lim_{T\to \infty}   \frac{1}{T }\Big[\frac14\int_{[0,T]^4}e^{-\theta\abs{t_1-s_1}-\theta\abs{t_2-s_2}}\abs{t_2-t_1}^{-\frac12} \abs{s_1-s_2}^{-\frac12}\dif\vec{ t}\dif \vec{s}-\frac{2}{\theta^2} T\log T] \\
&=-\frac{2}{\theta^2} + \lim_{T\to \infty}  \Big[ \int_{[0,T]^3}e^{-\theta\abs{t_1-s_1}-\theta(T-s_2)}(T-t_1)^{-\frac12} \abs{s_1-s_2}^{-\frac12}\dif { t}_1\dif \vec{s}-\frac{2}{\theta^2} \log T]\\
&\quad (\text{let} \quad x=T-s_2,\,y=T-s_1,\, z=T-t_1)\\
&=-\frac{2}{\theta^2} +\lim_{T\to \infty}  { \log T} \Big[\frac{1}{ \log T}\int_{[0,T]^3}e^{-\theta(x+\abs{y-z})}z^{-\frac12} \abs{x-y}^{-\frac12}\dif x\dif y\dif z-\frac{2}{\theta^2}\Big] \\
&=-\frac{2}{\theta^2}+ \sum_{i=1}^2  \lim_{T\to \infty}  {\log T} (J_i -\frac{1}{\theta^2}) + \sum_{i=3}^6  \lim_{T\to \infty}  J_i {\log T} ,
\end{align*} 
where 
\begin{align*}
J_i=\frac{1}{\log T}\int_{\Delta_i(T)}e^{-\theta(x+\abs{y-z})}z^{-\frac12} \abs{x-y}^{-\frac12}\dif x\dif y\dif z.
\end{align*}

Using a change of variable $u=y-x$, we have that 
\begin{align*}
 \log T (J_1 -\frac{1}{\theta^2}) &=\int_{0\le x\le y\le z\le T}e^{-\theta(x+z-y)}z^{-\frac12} \abs{x-y}^{-\frac12}\dif x\dif y\dif z -\frac{1}{\theta^2}\log T\\
 &=\int_{0\le u\le z\le T}e^{-\theta(z-u)}z^{-\frac12} u^{-\frac12} ( z-u)\dif u\dif z -\frac{1}{\theta^2}\log T\\
 &=\int_{0\le u\le z\le T}e^{-\theta(z-u)}\big(z^{\frac12} u^{-\frac12} -z^{-\frac12} u^{\frac12} \big)\dif u\dif z -\frac{1}{\theta^2}\log T.
 \end{align*}
 The formula of integration by parts implies that 
 \begin{align*}
 \int_{0\le u\le z\le T}e^{-\theta(z-u)}z^{\frac12} u^{-\frac12}\dif u\dif z&= \frac{T}{\theta}- \frac{1}{\theta} \frac{T^{\frac12}}{e^{\theta T}}\int_0^{T}u^{-\frac12}
e^{\theta u} \dif u+ \frac{1}{2\theta}\int_{0\le u\le z\le T}e^{-\theta(z-u)}z^{-\frac12} u^{-\frac12}\dif u\dif z,\\
 \int_{0\le u\le z\le T}e^{-\theta(z-u)}z^{-\frac12} u^{\frac12}\dif u\dif z&= \frac{T}{\theta}- \frac{1}{2\theta}\int_{0\le u\le z\le T}e^{-\theta(z-u)}z^{-\frac12} u^{-\frac12}\dif u\dif z,
\end{align*}
Using a change of variable $u=pz,\,p\in (0,1)$ and the integration by parts, we have that \begin{align*}
\int_{1}^{ T}\dif z\int_0^{z} e^{-\theta(z-u)}z^{-\frac12} u^{-\frac12}\dif u-\frac{1}{\theta}\log T&=\frac{1}{\theta} \int_{1}^{T} e^{-\theta z}z^{-1}\dif z\int_0^1 p\, \dif e^{\theta zp} -\frac{1}{\theta}\log T\\
&=\frac{1}{\theta^2}(1-\frac{1}{T^2}) -\frac{1}{\theta}\int_1^{T}e^{-\theta z}z^{-1}\dif z.
\end{align*}
Hence, we have that 
\begin{align*}
\lim_{T\to \infty}  \log T (J_1 -\frac{1}{\theta^2})= -\frac{1}{\theta^2}+\frac{1}{\theta^3}+\frac{1}{\theta} \int_{0\le u\le z\le 1}e^{-\theta(z-u)}z^{-\frac12} u^{-\frac12}\dif u\dif z-\frac{1}{\theta^2}\int_1^{\infty}e^{-\theta z}z^{-1}\dif z.
\end{align*}
Similarly, we have that 
\begin{align*}
 \lim_{T\to \infty}  {\log T} (J_2 -\frac{1}{\theta^2})&= \lim_{T\to \infty}   \big(\frac{1}{\theta}\int_{0\le u\le z\le T}e^{-\theta(z-u)}z^{-\frac12} u^{-\frac12}\dif u\dif z -\frac{1}{\theta^2}\log T\big)\\
 &= \frac{1}{\theta^3}+\frac{1}{\theta} \int_{0\le u\le z\le 1}e^{-\theta(z-u)}z^{-\frac12} u^{-\frac12}\dif u\dif z-\frac{1}{\theta^2}\int_1^{\infty}e^{-\theta z}z^{-1}\dif z,\\
 \lim_{T\to \infty}  J_i {\log T}&=\frac{{\pi}}{2\theta^2},\quad i=3,4,5,6 .
\end{align*}
Thus, the speed of $\frac{2\norm{f_T}^2_{\mathfrak{H}^{\otimes 2}}}{\log T}\to \frac{9\pi}{64\theta^3}$ is $1/\log T$ as $T\to \infty$.
\end{proof}

\noindent{\it Proof of Theorem~\ref{main thm}.\,} We only show the case of $H\in [\frac12,\,\frac34)$. The case of $H=\frac34$ is similar.
 
It follows from Theorem~\ref{kp},\, Lemma~\ref{bt exponential}  and Eq.(\ref{ratio 1})-(\ref{bt bt}) that  there exists a constant $C_{\theta, H}$ such that for $T$ large enough,
\begin{align}
&\sup_{z\in \Rnum}\abs{P(\sqrt{\frac{T}{\theta \sigma^2_H}} (\hat{\theta}_T-\theta )\le z)-P(Z\le z)}\le  \nonumber \\
& C_{\theta, H}\times\max\set{\abs{b_T^2-2\norm{f_T}^2},\,\norm{f_T\otimes_1 f_T},\, \norm{f_T\otimes_1 g_T},\,\innp{f_T,\,g_T},\,\norm{g_T}^2,\,\norm{g_T\otimes_1 g_T}}.\label{sup le}
\end{align}Denote $a=H\Gamma(2H)\theta^{-2H}$.
Lemma~\ref{bt exponential} and Lemma~\ref{2ft2} imply that  there exists a constant $c$ such that  for $T$ large enough,
\begin{align*}
\abs{b_T^2-2\norm{f_T}^2}\le \abs{b_T^2- a^2}+\abs{2\norm{f_T}^2-a^2}\le c\times \frac{1}{T^{3-4H}}.
\end{align*}
Lemma~\ref{gt gt} implies that  there exists a constant $c$ such that  for $T$ large enough,
\begin{align*}
\norm{f_T\otimes_1 g_T},\,\innp{f_T,\,g_T},\,\norm{g_T\otimes_1 g_T}  \le c\times \frac{1}{\sqrt{T}},\qquad
\norm{g_T}^2\le c\times \frac{1}{{T}}.
\end{align*}

Substituting (\ref{zhou ineq}) and the above inequalities into (\ref{sup le}), we obtain the desired Berry-Ess\'{e}en bound (\ref{hle 34}).  
{\hfill\large{$\Box$}}\\

\vskip 0.2cm {\small {\bf  Acknowledgements}:
We would like to gratefully thank the referee for very valuable suggestions which lead to the improvement of the new version. 
 Y. Chen is supported by NSFC (No.11871079).
}



\end{document}